\documentclass[11pt,twoside,leqno,a4paper]{amsart}

\usepackage{amssymb}
\usepackage{enumerate}

\newtheorem{thm}{Theorem}
\newtheorem{cor}{Corollary}
\newtheorem{defn}{Definition}

\theoremstyle{definition}
\theoremstyle{remark}
\newtheorem{rem}{Remark}
\DeclareMathOperator{\V}{V}
\newcommand{\D}{d}
\DeclareMathOperator{\f}{\varphi}
\newcommand{\e}{e}
\DeclareMathOperator{\cl}{cl}
\DeclareMathOperator{\h}{\chi}
\DeclareMathOperator{\N}{N}
\DeclareMathOperator{\G}{G}

\newcommand{\abs}[1]{\left\vert#1\right\vert}
\newcommand{\set}[1]{\left\{#1\right\}}

\everymath{\displaystyle}
\begin{document}

\title[An Inequality for Generalized Chromatic Graphs]
 {An Inequality for Generalized Chromatic Graphs}

\author{Asen Bojilov}

\address{Faculty of Mathematics and Informatics,
University of Sofia, Bulgaria}

\email{bojilov@fmi.uni-sofia.bg}

\author{Nedyalko Nenov}

\address{Faculty of Mathematics and Informatics,
University of Sofia, Bulgaria}

\email{nenov@fmi.uni-sofia.bg}

\thanks{This work was supported by the Scientific Research Fund
of the St. Kliment Ohridski Sofia University under contract No187, 2011.}




\dedicatory{}



\begin{abstract}
Let $G$ be a simple $n$-vertex graph with degree sequence
$d_1,d_2,\dots,d_n$ and vertex set $\V(G)$. The degree of
$v\in\V(G)$ is denoted by $\D(v)$. The smallest integer $r$
for which $\V(G)$ has an $r$-partition
$$
\V(G)=V_1\cup V_2\cup\dots\cup V_r,\quad V_i\cap V_j=\emptyset,
\quad,i\neq j
$$
such that $\D(v)\leq n-\abs{V_i}$, $\forall v\in V_i$, $i=1,2,\dots,r$
is denoted by $\f(G)$. In this note we prove the inequality
$$
\f(G)\geq\frac n{n-\bar{\bar{d}}},
$$
where $\bar{\bar{d}}=\sqrt{\dfrac{d_1^2+d_2^2+\cdots+d_n^2}n}$.
\end{abstract}

\maketitle

\section{Introduction}

We consider only finite, non-oriented graphs without loops and
multiple edges. We shall use the following notations:

$\V(G)$ -- the vertex set of $G$;

$\e(G)$ -- the number of edges of $G$;

$\cl(G)$ -- the clique number of $G$;

$\h(G)$ -- the chromatic number of $G$;

$\N(v)$, $v\in\V(G)$ --- the set of neighbours of a vertex $v$;

$\N(v_1,v_2,\dots,v_k)=\bigcap_{i=1}^k\N(v_i)$;

$\D(v)$ -- the degree of a vertex $v$;

$\G[V]$, $V\subseteq\V(G)$ -- induced subgraph by $V$.

\begin{defn}
Let $G$ be a graph, $\abs{\V(G)}=n$ and $V\subseteq\V(G)$.
The set $V$ is called a $\delta$-set in $G$, if
$$
\D(v)\leq n-\abs{V}\text{ for all $v\in V$.}
$$
\end{defn}

Clearly, any independent set $V$ of vertices of a graph $G$ is a
$\delta$-set in $G$ since
$\N(v)\subseteq\V(G)\setminus V$ for all $v\in V$.
It is obvious that if $V\subseteq\V(G)$ and
$\abs V\geq\max\set{\D(v)\mid v\in\V(G)}$ then
$\V(G)\setminus V$ is a $\delta$-set in $G$
(it is possible that $\V(G)\setminus V$ is not independent).

\begin{defn}
A graph $G$ is called a generalized $r$-partite graph if
there is a $r$-partition
$$
\V(G)=V_1\cup V_2\cup\dots\cup V_r,\quad V_i\cap V_j=\emptyset,
\quad,i\neq j
$$
where the sets $V_1,V_2,\dots,V_r$ are $\delta$-sets in $G$.
The smallest integer $r$ such that $G$ is a generalized $r$-partite
is denoted by $\f(G)$.
\end{defn}

As any independent vertex set of $G$ is a $\delta$-set in $G$,
we have $\f(G)\leq\h(G)$.
In fact, the following stronger inequality \cite{N2006}
\begin{equation}\label{eq1.2}
\f(G)\leq\cl(G)
\end{equation}
holds.

Let $\V(G)=\set{v_1,v_2,\dots,v_n}$ and $\cl(G)=r$.
Define
$$
\bar{d}=\frac{\D(v_1)+\D(v_2)+\cdots+\D(v_n)}n, \qquad
\bar{\bar{d}}=\sqrt{\frac{\D^2(v_1)+\D^2(v_2)+\cdots+\D^2(v_n)}n}.
$$

By the classical Turan Theorem, \cite{Tur} (see also \cite{Khad1977})
we have
\begin{equation}\label{eq1.3}
\e(G)\leq\frac{n^2(r-1)}{2r}.
\end{equation}

The equality in \eqref{eq1.3} holds if and only if
$n\equiv0\pmod r$ and $G$ is complete $r$-chromatic and regular.

It is proved in \cite{KhadTur} that
\begin{equation}\label{eq1.4}
\e(G)\leq\frac{n^2(\f(G)-1)}{2\f(G)}.
\end{equation}

According to \eqref{eq1.2} the inequality \eqref{eq1.4} is stronger
than the inequality \eqref{eq1.3}. But in case of equality in \eqref{eq1.4}
the graph $G$ is not unique as it is in the Turan theorem.

Since $\bar{d}(G)=\frac{2\e(G)}n$, it follows from \eqref{eq1.4} that
\begin{equation}\label{eq1.5}
\f(G)\geq\frac n{n-\bar{d}(G)}.
\end{equation}
In this note we give the following improvement of the inequality
\eqref{eq1.5}.

\begin{thm}\label{thm1.1}
Let $G$ be a $n$-vertex graph. Then
\begin{equation}\label{eq1.6}
\f(G)\geq\frac n{n-\bar{\bar{d}}(G)}.
\end{equation}
The equality in \eqref{eq1.6} holds if and only if
$n\equiv0\pmod{\f(G)}$ and $G$ is regular graph of degree
$\frac{n(\f(G)-1)}{\f(G)}$.
\end{thm}


\section{Auxiliary results}

We denote the elementary symmetric polynomial of degree $s$ by
$\sigma_s(x_1,x_2,\dots,x_n)$, $1\leq s\leq n$, i.\,e.
$$
\sigma_s(x_1,x_2,\dots,x_n)=x_1x_2\dots x_s+\cdots.
$$

Further, we will use the following equalities:
\begin{align}
x_1^2+x_2^2+\cdots+x_n^2&=\sigma_1^2-2\sigma_2,\label{eq2.7}\\%
x_1^3+x_2^3+\cdots+x_n^3&=\sigma_1^3-3\sigma_1\sigma_2+3\sigma_3,\label{eq2.8}%
\end{align}
where $\sigma_i=\sigma_i(x_1,x_2,\dots,x_n)$.

In order to prove \ref{thm1.1} we will use the following well known inequality
(particular case of the Maclaurin inequality, see ,\cite{HL},\cite{Extr}).

\begin{thm}\label{thm2}
Let $x_1,x_2,\dots,x_n$ be non-negative reals and
$\sigma_s(x_1,x_2,\dots,x_n)=\sigma_s$. Then
\begin{equation}\label{eq2.9}
\sqrt[s]{\frac{\sigma_s}{\binom ns}}\leq
\frac{x_1+x_2+\cdots+x_n}n=\frac{\sigma_1}n,\quad
1\leq s \leq n.
\end{equation}
If $s\geq2$, then the equality in \eqref{eq2.9} holds iff
$x_1=x_2=\cdots=x_n$.
\end{thm}

A straight and very short prove of Theorem \ref{thm2} is given in
\cite{Khadm}.


\section{Proof of the Theorem~\ref{thm1.1}}

Let $\f(G)=r$, $\V(G)=\set{v_1,v_2,\dots,v_n}$ and
\begin{equation}\label{eq2.10}
\V(G)=V_1\cup V_2\cup\dots\cup V_r,\quad
V_i\cap V_j=\emptyset,\quad i\neq j,
\end{equation}
where $V_1,V_2,\dots,V_r$ are $\delta$-sets in $G$, i.\,e.
if $n_i=\abs{V_i}$, $i=1$, 2,\dots, $r$ then
\begin{equation}\label{eq2.11}
\D(v)\leq n-n_i,\quad\forall v\in V_i.
\end{equation}
It follows from \eqref{eq2.10} that
$$
\D^2(v_1)+\D^2(v_2)+\cdots+\D^2(v_n)=
\sum_{i=1}^r\sum_{v\in V_i}\D^2(v).
$$
According to \eqref{eq2.11}
$$
\sum_{v\in V_i}\D^2(v)\leq n_i(n-n_i)^2.
$$
Thus we have
$$
\D^2(v_1)+\D^2(v_2)+\cdots+\D^2(v_n)\leq
\sum_{i=1}^rn_i(n-n_i)^2.
$$
From \eqref{eq2.7} and \eqref{eq2.8} we see that
$$
\sum_{i=1}^rn_i(n-n_i)^2=n\sigma_2+3\sigma_3,
$$
where $\sigma_2=\sigma_2(n_1,n_2,\dots,n_r)$,
$\sigma_3=\sigma_3(n_1,n_2,\dots,n_r)$.

Thus we obtain the inequality
\begin{equation}\label{eq2.12}
\D^2(v_1)+\D^2(v_2)+\cdots+\D^2(v_n)\leq
n\sigma_2+3\sigma_3.
\end{equation}
Since $\sigma_1=n$, Theorem~\ref{thm2} yields
\begin{equation}\label{eq2.13}
\sigma_2\leq\frac{n^2(r-1)}{2r}\text{\ \ and\ \ }
\sigma_3\leq\frac{n^3(r-1)(r-2)}{6r^2}.
\end{equation}
Now the inequality \eqref{eq1.6} follows from \eqref{eq2.12} and \eqref{eq2.13}.

Obviously, if $n\equiv0\pmod r$ and
$\D(v_1)=\D(v_2)=\cdots=\D(v_r)=\frac{n(r-1)}r$ we have equality
in \eqref{eq1.6}. Now, let us suppose that we have equality in inequality
\eqref{eq1.6}. Then we have equality in \eqref{eq2.13} and
\eqref{eq2.11} too.
From $r=\f(G)=\frac n{n-\bar{\bar{d}}}$ it is clear that
$r$ divides $n$. By Theorem~\ref{thm2}, we have
$$
n_1=n_2=\cdots=n_r=\frac nr.
$$
Because of the equality in \eqref{eq2.11}, i.\,e. $\D(v)=n-n_i$,
$v\in V_i$, we have
$$
\D(v_1)=\D(v_2)=\cdots=\D(v_r)=\frac{n(r-1)}r.
$$
Theorem \ref{thm1.1} is proved.

\section{Some corollaries}

\begin{defn}[\cite{Khad1977}]
Let $G$ be a graph and $v_1,v_2,\dots,v_r\in\V(G)$. The sequence
$v_1,v_2,\dots,v_r$ is called an $\alpha$-sequence in $G$ if the
following conditions are satisfied:
\begin{enumerate}[(i)]
\item $\D(v_1)=\max\set{\D(v)\mid v\in|\V(G)}$;
\item $v_i\in\N[v_1,v_2,\dots,v_{i-1}]$ and $v_i$ has maximal
degree in the induced subgraph $\G[\N(v_1,v_2,\dots,v_{i-1}]$,
$2\leq i\leq r$.
\end{enumerate}
\end{defn}

\begin{defn}
Let $G$ be a graph and $v_1,v_2,\dots,v_r\in\V(G)$. The sequence
$v_1,v_2,\dots,v_r$ is called a $\beta$-sequence in $G$ if the
following conditions are satisfied:
\begin{enumerate}[(i)]
\item $\D(v_1)=\max\set{\D(v)\mid v\in|\V(G)}$;
\item $v_i\in\N(v_1,v_2,\dots,v_{i-1})$ and
$\D(v_i)=\max\set{\D(v)\mid v\in\N(v_1,v_2,\dots,v_{i-1})}$,
$2\leq i\leq r$.
\end{enumerate}
\end{defn}

\begin{cor}
Let $v_1,v_2,\dots,v_r$, $r\geq2$ be an $\alpha$- or a $\beta$-sequence
in an $n$-vertex graph $G$ such that $\N(v_1,v_2,\dots,v_r)$ is a
$\delta$-set. Then
\begin{equation}\label{eq4.14}
r\geq\frac n{n-\bar{\bar{d}}}.
\end{equation}
\end{cor}

\begin{proof}
Since $\N(v_1,v_2,\dots,v_ð)$ is a $\delta$-set, $G$ is a generalized
$r$-partite graph, \cite{Khad2005}. Thus, $r\geq\f(G)$ and \eqref{eq4.14}
follows from Theorem \ref{thm1.1}.
\end{proof}

\begin{cor}
Let $v_1,v_2,\dots,v_r$, $r\geq2$, be a $\beta$-sequence
in $n$-vertex graph $G$ such that
\begin{equation}\label{eq4.15}
\D(v_1)+\D(v_2)+\cdots+\D(v_r)\leq(r-1)n.
\end{equation}
Then the inequality \eqref{eq4.14} holds.
\end{cor}

\begin{proof}
From \eqref{eq4.15} it follows that $G$ is a generalized $r$-partite
graph (\cite{KhadSat},\cite{KhadMax}).
\end{proof}

The next corollary follows from \eqref{eq1.2} and Theorem \ref{thm1.1}.

\begin{cor}[\cite{Edw}]
Let $G$ be an $n$-vertex graph. Then
\begin{equation}\label{eq4.16}
\cl(G)\geq\frac n{n-\bar{\bar{d}}}.
\end{equation}
\end{cor}

\begin{rem}
The prove of the inequality \eqref{eq4.16} given in \cite{Edw}
is incorrect, since the arguments on p.53, rows 8 and 9 from top,
is not valid.
\end{rem}

\begin{cor}
Let $G$ be an $n$-vertex graph such that
\begin{equation}\label{eq4.17}
\cl(G)=\frac n{n-\bar{\bar{d}}}.
\end{equation}
Then $G$ is regular and complete $\cl(G)$-chromatic graph.
\end{cor}

\begin{proof}
Let $\f(G)=r$. By \eqref{eq4.17}, \eqref{eq1.2} and Theorem \ref{thm1.1}
we have
$$
\cl(G)=\f(G)=r=\frac n{n-\bar{\bar{d}}}.
$$
By Theorem \ref{thm1.1}, $n\equiv0\pmod r$ and $G$ is a regular graph
of degree $\frac{n(r-1)}r$. Thus
$$
\e(G)=\frac{n^2(r-1)}{2r}=\frac{n^2(\cl(G)-1)}{2\cl(G)}.
$$
According to Turan Theorem, $G$ is complete $r$-chromatic and regular.
\end{proof}

\bibliographystyle{amsplain}

\begin{thebibliography}{10}

\bibitem{Edw}
C.~Edwards and C.~Elphick, \emph{Lower bounds for the clique and the chromatic
  number of a graph}, Discrete Appl. Math. (1983), no.~5.

\bibitem{HL}
G.~H. Hardy, J.~F. Litelewood, and G.~Polya, \emph{Inequalities}, 1934.

\bibitem{Extr}
N.~Khadzhiivanov, \emph{Extremal theory of graphs}, Sofia University, Sofia,
  1990, (in Bulgarian).

\bibitem{Khadm}
N.~Khadzhiivanov and N.Nenov, \emph{An equalities for elementary symetric
  functions}, Matematica (1977), no.~4, (in Bulgarian).

\bibitem{Khad1977}
\bysame, \emph{Extremal problems for $s$-graphs and a theorem of {T}uran},
  Serdica Math. J. (1977), no.~3, (in Russian).

\bibitem{KhadTur}
\bysame, \emph{Generalized $r$-partite graphs and turan's theorem}, Compt.
  rend. Acad. bulg. Sci. \textbf{57} (2004).

\bibitem{KhadSat}
\bysame, \emph{Saturated $\beta$-sequences in graphs}, Compt. rend. Acad. bulg.
  Sci. \textbf{57} (2004).

\bibitem{KhadMax}
\bysame, \emph{Sequence of maximal degree verteces in graphs}, Serdica Math. J.
  (2004), no.~30.

\bibitem{Khad2005}
\bysame, \emph{Balanced vertex sets in graphs}, Ann. Univ. Sofia, Fac. Math.
  Inf. \textbf{97} (2005).

\bibitem{N2006}
N.Nenov, \emph{Improvement of graph theory {W}ei's inequlity}, Mathematics and
  education in mathematics (2006), 191--194, Proceedings of the Thirty Fifth
  Spring Conference of Union of Bulgarian Mathematics, Borovets, April 5-8,
  2006.

\bibitem{Tur}
P.~Turan, \emph{Eine extremalaufgabe aus der graphentheorie}, Mat. Fiz. Lapok
  \textbf{48} (1941).

\end{thebibliography}

\end{document}